\documentclass[11pt,reqno]{amsart}

\usepackage[colorlinks=true, linkcolor=blue, anchorcolor=blue, citecolor=blue, filecolor=blue, menucolor= blue, urlcolor=blue]{hyperref}

\usepackage{amsmath,amssymb,amsthm,fullpage}
\usepackage{verbatim}
\usepackage{comment}
\usepackage{multirow}
\usepackage{mathtools}
\usepackage{paralist}
\usepackage{enumitem}

\makeatletter
\renewcommand*\env@matrix[1][*\c@MaxMatrixCols c]{%
  \hskip -\arraycolsep
  \let\@ifnextchar\new@ifnextchar
  \array{#1}}
\makeatother

\long\def\eatit#1{}
\newtheorem{thm}{Theorem}[section]
\newtheorem{prop}[thm]{Proposition}
\newtheorem{lem}[thm]{Lemma}
\newtheorem{cor}[thm]{Corollary}

\theoremstyle{definition}
\newtheorem{defn}[thm]{Definition}

\newtheorem{rem}[thm]{Remark}

\newcommand{\p}{\mathfrak p}
\newcommand{\pr}[1]{{{\bf P}^{#1}}}

\newcommand{\field}{K}
\newcommand{\lra}{\longrightarrow}
\newcommand{\ra}{\rightarrow}
\newcommand{\rees}{\mathcal{R}}

\newcommand{\srees}{\mathcal{R}_s}

\DeclareMathOperator{\Sym}{Sym}

\DeclareMathOperator{\reg}{reg}

\newcommand{\XX}{{\mathbf X}}
\newcommand{\ZZ}{{\mathbb Z}}

\title[Symbolic Rees algebra for Fermat configurations]{Ordinary and  Symbolic Rees Algebras\\ for Ideals of  Fermat Point Configurations}
\author{Uwe Nagel and Alexandra Seceleanu}
\thanks{ The first author was partially supported by the National Security Agency under Grant Number H98230-12-1-0247 and by the Simons Foundation under grant  \#317096. The second author  was partially supported by an NSF-AWM mentoring travel grant.}

\begin{document}
\maketitle


\begin{abstract}
Fermat ideals define planar point configurations that are closely related to the intersection locus of the members of a specific pencil of curves. These ideals have gained recent popularity as counterexamples to some proposed containments between symbolic and ordinary powers \cite{refDST}.
We give a systematic treatment of the family of Fermat ideals, describing explicitly the minimal generators and the minimal free resolutions of all their ordinary powers as well as many symbolic powers. We use these to study the ordinary and the symbolic Rees algebra of Fermat ideals. Speci\-fically, we show  that the symbolic Rees algebras of Fermat ideals are Noetherian. Along the way, we give formulas for the Castelnuovo-Mumford regularity of powers of Fermat ideals and we determine their reduction ideals.
\end{abstract}


\section{Introduction}

Let $n\geq 2$ be  an integer, let $\field$ be a field that contains $n$ distinct $n^{\rm th}$ roots of 1, and consider the ideal $$I=(x(y^n-z^n),y(z^n-x^n),z(x^n-y^n))\subset R=\field[x,y,z],$$ which we shall refer to as a {\em Fermat} ideal. The variety described by this ideal is a reduced set of $n^2+3$ points in $\pr{2}$, as shown in \cite[Proposition 2.1]{refHS}. Specifically, $n^2$ of these points form the intersection locus of the pencil of curves spanned by $x^n-y^n$ and $x^n-z^n$, while the other 3 are the coordinate points $[1:0:0], [0:1:0]$ and $[0:0:1]$. The general member of the pencil is isomorphic to the Fermat curve $x^n+y^n+z^n$, justifying the terminology. When $n$ is allowed to vary we obtain an infinite collection of distinct Fermat ideals, which we refer to as the Fermat family of ideals.

The goal of this note is to understand the nature of the ordinary powers and  symbolic powers of Fermat ideals. Our motivation stems from the work of \cite{refDST} and \cite{refHS}, where it is shown that the relation between ordinary and symbolic powers of Fermat ideals is quite surprising.  In \cite{refHH}, motivated by some deep results of \cite{refELS,refHoHu} and also by the fact that this is true for general points as shown in \cite{refBH}, it was asked whether all ideals $I$ defining reduced sets of planar points satisfy the containment  $I^{(3)}\subseteq I^2$. The surprising occurrence of a non-containment $I^{(3)}\not\subseteq I^2$ was first discovered by \cite{refDST} for the simplest case of the Fermat ideal with $n=3$. Later,  in \cite{refHS} the same non-containment was observed to extend to the entire family of Fermat ideals and in \cite{refS},  the minimal free resolutions of the second and third powers of $I$ were used to give an alternate justification for the  non-containment. 

Rather than focusing on specific ordinary or symbolic powers, in this paper we take a global approach by means of  assembling all powers of these ideals into bi-graded algebras. Our main objects of interest  are the \emph{Rees algebra} of $I$  defined as $\rees(I)=\oplus_{i\geq0} I^it^i$ and the \emph{symbolic Rees algebra} of $I$ given by $\srees(I)=\bigoplus_{i\geq0} I^{(i)}t^i$, respectively. 
We study the properties of these Rees algebras, often in close connection with the homological properties of the various powers of Fermat ideals.

Our paper is organized as follows. In section \ref{sect2} we prove that the Rees algebra in the case of the Fermat ideals  is as simple as possible, namely they have linear type. We use this to derive an explicit formula for the minimal free resolutions of all ordinary powers of Fermat ideals. In stark contrast to the Rees algebra, the symbolic Rees algebra of a homogeneous ideal may in general not be Noetherian, even for ideals of points. Perhaps the most famous illustration of this phenomenon is given by Nagata in  \cite{refNa}, where he constructs a counterexample to Hilbert's fourteenth problem.  Although criteria that force symbolic Rees algebras of certain ideals to be finitely generated have been given (see \cite{refSch} or \cite{refHu}), not many interesting examples of such ideals that represent geometric collections of points are known. In section \ref{sec:symb Rees} we show that, in the case of Fermat ideals, the symbolic Rees algebra is Noetherian. It follows in particular that a sufficient condition for a symbolic Rees algebra being Noetherian established in \cite{refHH} is not necessary. In order to obtain our result on symbolic Rees algebras, we need to completely determine the minimal generators and the minimal free resolutions of certain families of non-reduced ideals (fat points) supported at the points of a Fermat configuration. These  results form the technical core of the paper and take up the bulk of section \ref{sect3}. Furthermore, they allow us to determine the Castelnuovo-Mumford regularity for all ordinary powers and most symbolic powers of Fermat ideals. As another application we provide some explicit minimal homogeneous reductions and show that they have reduction number one in section \ref{sect5}.


\section{The Rees algebra of the Fermat ideals and resolutions of ordinary powers}\label{sect2}

In the following, we employ the terminology {\em almost complete intersection} to mean an ideal minimally generated by a set of generators that has cardinality at most one higher than the height of the ideal.  We call {\em strict almost complete intersections} those ideals minimally generated by a set of generators that has cardinality exactly one higher than the height of the ideal.   Note that our Fermat ideals are strict almost complete intersections. 

We start by recalling  the description of the ordinary and symbolic Rees algebras.   
 
 \begin{defn} 
         \label{defRees}
 Denote  by $I=(f_1, f_2,\ldots, f_n)\subseteq R$ an $n$-generated homogeneous ideal of a polynomial ring $R$. Let $S=R[T_1,T_2,\ldots,T_n]$ denote a bigraded polynomial ring where the variables of $R$ have degree $(1,0)$ and the variables $T_i$ have degree $(\deg(f_i),1)$.
 The $R$-algebra epimorphism  $R[T_1,T_2,\ldots,T_3] \to \rees(I)$ sending $T_i\mapsto f_it$ gives presentations of the  symmetric algebra $\Sym(I)$ and Rees algebra $\rees(I)$ respectively as quotients of the bigraded polynomial ring $S=R[T_1,T_2,\ldots,T_3]$.  Writing $L$ for the kernel of this epimorphism yields:
$$ \rees(I)=S/L, \mbox{ where } L=\left\{F(T_1,T_2,\ldots,T_n):F(f_1,f_2,\ldots,f_n)=0\right\}.$$ 
In turn, the presentation of the symmetric algebra of $I$ only takes into account the bidegree $(*,1)$ relations between generators of $I$:
$$\Sym(I)=S/L_1, \mbox{ where } L_1=\left\{\sum_{i=1}^n b_iT_i : \sum_{i=1}^n b_if_i=0\right\}.$$
\end{defn}
The structure of these algebras does not depend on the set $f_1,\ldots, f_n$ of minimal generators of $I$ chosen, but only on $I$ itself. Furthermore, there is a canonical graded surjection $\Sym(I)\twoheadrightarrow \rees(I)$.

\begin{defn}
If for some ideal $I$ there is an isomorphism  $\Sym(I)\simeq \rees(I)$, $I$ is said to have linear type.   Equivalently, $I$ has linear type if and only if  the  ideal of equations of the Rees algebra  is generated in bidegree $(*,1)$. In the notation of Definition \ref{defRees}, if $L$ is the ideal of relations for $\rees(I)$, this means that $L=L_1$. 
\end{defn}

We start with a structural result about the Rees algebra of almost complete intersections which define reduced sets of points. 

\begin{lem}\label{lintype}
Let $I$ be an almost complete intersection ideal defining a reduced set of points in $\pr{N}$. Then $I$ is an ideal of linear type.
\end{lem}
\begin{proof}
 It is shown in \cite[Corollary 5.65]{refVa} that an almost complete intersection $I$ of height $h$ that is a generic complete intersection (i.e. $I$ localized at each of its associated primes of codimension $h$  is a complete intersection),  is an ideal of linear type.
 All  these conditions are clearly satisfied in case $I$ defines a reduced set of points since $R/I$ is height 2 arithmetically Cohen-Macaulay, with $Ass(I)=\{I_{p_i}\ |  1\leq i\leq e(R/I)\}$ and  $I_{p_i}$ is minimally generated by $N$ linear forms that form a regular sequence for every ideal  defining one of the points $p_i$ in the given set.
\end{proof}

 \begin{cor}\label{CI}
 The Rees algebra of the Fermat ideal $I=(x(y^n-z^n),y(z^n-x^n),z(x^n-y^n))$is a complete intersection whose defining ideal is generated by two forms of bidegree $(n+3,1)$ and $(2n,1)$.
  \end{cor}
 \begin{proof}
By Lemma \ref{lintype}, the Rees algebra of a Fermat ideal is isomorphic to its symmetric algebra. Next we show the latter is a complete intersection. Let $A=\left[\begin{matrix}P_1 & P_2 & P_3 \\ Q_1 & Q_2 & Q_3 \end{matrix}\right]^T$ be a presentation matrix for the module of syzygies on $I$. By the proof of \cite[Theorem 2.1]{refDHNSST}, we have $\deg P_i=2$ and $\deg Q_i=n-1$. As quotients of the polynomial ring $S=R[T_1,T_2,T_3]$ the symmetric and Rees algebra of $I$ are then defined by 
 $$\rees(I)\simeq \Sym(I) \simeq S/(P_1T_1+P_2T_2+ P_3T_3, Q_1T_1+Q_2T_2+Q_3T_3).$$
Since the two syzygies of $I$ are algebraically independent, the height of this ideal is two. This yields the desired conclusion.
\end{proof}
 
 A prevailing technique (\cite{refKo}, \cite{refCHT}) used in investigating resolutions of the powers of $I$ relies on using the resolution of the Rees algebra. Consider the bihomogeneous minimal free resolution of $\rees(I)$:
$$ 0 \lra \bigoplus_{(i,j)} S(-i,-j)^{\beta_{p,(i,j)}} \lra \ldots \lra \bigoplus_{(i,j)} S(-i,-j)^{\beta_{1,(i,j)}} \lra S \lra \rees(I) \lra 0. $$
Note that $\rees(I)_{(*,r)}\simeq I^r$ as $R$-modules via the map $T_i\mapsto f_i$. Restricting to the strand of this resolution corresponding to the $R$-submodule of the resolvent $S$ above consisting of elements of bidegrees $(*,r)$  yieds a (not necessarily minimal) free resolution of $I^r$ over $R$ as follows:
$$ 0 \lra \bigoplus_{(i,j)} S(-i,-j)_{(*,r)}^{\beta_{p,(i,j)}} \lra \ldots \lra \bigoplus_{(i,j)} S(-i,-j)_{(*,r)}^{\beta_{1,(i,j)}} \lra S_{(*,r)} \ra I^r \ra 0. $$ 

\begin{thm}\label{res}
Let $I$ be a strict almost complete intersection ideal with minimal generators of the same degree $d$ defining a reduced set of points in $\pr{2}$. Assume that the module of syzygies on $I$ is generated in degrees $d_0$ and $d_1$. Then the minimal free resolution of $I^r$ over $R=\field[x,y,z]$ is 
$$0 \ra R(-(r+2)d)^{\binom {r}{2}} \ra\begin{matrix}R(-(r+1)d-d_1)^{\binom{r+1}{2}}\\
 \oplus \\
 R(-(r+1)d-d_0)^{\binom{r+1}{2}} 
 \end{matrix} \ra R(-rd)^{\binom{r+2}{2}} \ra I^r \ra 0.$$
\end{thm}
\begin{proof}
As in Corollary \ref{CI},  $\rees(I)$ is a complete intersection generated in bidegrees $(d+d_0,1)$ and $(d+d_1,1)$.  Recall that the degree of the minimal syzygies in a Hilbert-Burch resolution are related by $d_0+d_1=d$. Resolving the complete intersection $\rees(I)$ over $S$ we obtain:
\begin{equation*}\label{reesres}
0 \lra S(-2d,-2) \lra S(-d-d_1,-1) \oplus S(-d-d_0,-1) \lra S \lra \rees(I) \lra 0.
\end{equation*}

Taking the strand of degree $(*,r)$ of this complex and keeping in mind the following identities
\begin{eqnarray*}
S_{(*,r)} &=& \bigoplus_{\sum a_i=r}(\prod_{j=1}^n T_j^{a_i})R(-rd)\simeq R(-rd)^{\binom{r+n-1}{r}} \mbox{ and}\\ S(-i,-j)_{(*,r)} &=& \bigoplus_{\sum a_i=r-j}(\prod_{j=1}^n T_j^{a_i})R(-rd-i)\simeq R(-rd)^{\binom{r+n-1-j}{r}},\\
\end{eqnarray*}
yields for 3-generated ideals $I$ a free resolution of $I^r$ over $R$ of the form 
\begin{equation*}
0 \ra R(-(r+1)d)^{\binom {r}{2}} {\ra} 
\begin{matrix}R(-rd-d_1)^{\binom{r+1}{2}}\\
 \oplus \\
 R(-rd-d_0)^{\binom{r+1}{2}} 
 \end{matrix}{\ra} R(-rd)^{\binom{r+2}{2}} \ra I^r \ra 0.
\end{equation*}

 Although this is not generally the case, the resolution above is in fact minimal as long as the $\binom{r+2}{2}$ obvious generators of $I^r$ form a minimal generating set, because the consecutive terms appear with distinct shifts, therefore there can be no cancellations. However, the fact that all $\binom{r+2}{2}$ obvious generators are needed to generate $I^r$ follows in the case where $I$ is of linear type from the fact that there are no elements of bidegree $(0,r)$ in the defining ideal of $\rees(I)$, which would be forced by this type of nonminimality.  Note also that the binomial coefficient ${\binom {r}{2}}$ is 0 if and only if $r=1$, thus $I$ is the only ordinary power that is a perfect ideal. 
\end{proof}

\begin{cor}\label{cor:res-ordinary}
The minimal free resolutions of the ordinary powers of the Fermat ideal $$I=(x(y^n-z^n),y(z^n-x^n),z(x^n-y^n))$$ \vspace{-0.5em}
are:
\begin{itemize}
\item if $r=1$
\[
\begin{small}
0 \ra
R(-2n) 
\oplus 
R(-n-3)
 \ra R(-n-1)^{3} \ra I \ra 0.
 \end{small}
 \]
\item if $r\geq 2$
\[
\begin{small}
0 \ra R(-(r+1)(n+1))^{\binom {r}{2}} \ra \begin{matrix}
R(-r(n+1)-n+1)^{\binom{r+1}{2}} \\
\oplus \\
R(-r(n+1)-2)^{\binom{r+1}{2}}\end{matrix}
 \ra R(-r(n+1))^{\binom{r+2}{2}} \ra I^r \ra 0.
 \end{small}
 \]
 \end{itemize}
 
 The Castelnuovo-Mumford regularity of the ordinary powers of Fermat ideals is given by
$$
\reg(I^r)=\begin{cases}
2n & \text{if } r=1\\
rn+r+n-1 & \text{if }   r\ge 2.
\end{cases}
$$

\end{cor}
\begin{proof}
The minimal free resolution of $I$ (the case $r=1$) can be found in the proof of \cite[Theorem 2.1]{refDHNSST}. The minimal free resolutions for the higher powers ($r\geq 2$) follow by setting $d=n+1,d_0=2, d_1=n-1$ in Theorem \ref{res}. The graded shifts in these resolutions justify the regularity.
\end{proof}


\section{Symbolic Powers of Fermat ideals}\label{sect3}

We now establish properties of symbolic powers of Fermat ideals. This includes a description of their minimal generators and their graded minimal free resolutions. In order to achieve this we need to study ideals of a larger class of fat points, all supported on Fermat configurations. 

\subsection{Resolutions of symbolic powers}
Recall that  $I=(x(y^n-z^n),y(z^n-x^n),z(x^n-y^n))$  is the ideal of a Fermat configuration of $n^2+3$ (reduced) points in $\pr{2}$. By the classical Nagata-Zariski theorem \cite[Theorem 3.14]{refEis}, the $m$-th symbolic power of $I$, $I^{(m)}$ is the set of homogeneous polynomials that vanish to order at least $m$ at every point in the zero locus of $I$. Algebraically, since $I$ can be written as 
\[
I=(x^n-y^n,y^n-z^n)\cap(x,y)\cap(y,z)\cap(z,x)
\]
and each of the ideals listed in this decomposition of $I$ is generated by a regular sequence (such ideals have their symbolic powers equal to their respective ordinary powers) it follows that
\begin{equation} 
    \label{eq:decomp I}
 I^{(m)}=(x^n-y^n,y^n-z^n)^m\cap(x,y)^m\cap(y,z)^m\cap(z,x)^m.
\end{equation}

Although this description of the symbolic powers has the advantage of being concise, it is not best suited for studying the fine relationship between various symbolic powers. The approach we take in this section is to  exhibit explicit minimal generators and minimal free resolutions for some of the symbolic powers of $I$. Since the symbolic powers are prefect ideals of height two, this is equivalent to describing a Hilbert-Burch matrix  corresponding to each of these ideals. We build these Hilbert-Burch matrices as block matrices with some of the blocks of the form indicated below.

\begin{defn}\label{def:blocks}
For integers $0\leq j \leq  t$ and elements $a,b$ of a commutative ring $R$, we define the following matrices and column vectors:
\begin{itemize}
\item $H(a,b)_t=\left[\begin{matrix}-b & 0 & \ldots &  0 \\ a & -b & \ldots &  0\\ \vdots &\vdots &&  \vdots\\  0 & 0 & \ldots  & -b\\ 0 & 0 & \ldots  & a \end{matrix}\right]\in\mathcal{M}_{(t+1)\times t}( R )$
\item $C(a,b)_t=\left[\begin{matrix} a & -b & 0 & \ldots &  0 \\ 0 & a & -b & \ldots &  0\\ \vdots &\vdots &&  \vdots\\ 0 &  0 & 0 & \ldots  & -b\\ -b & 0 & 0 & \ldots  & a \end{matrix}\right]\in\mathcal{M}_{t\times t}( R )$
\item $E_j \in \ZZ^{j+1}$  is the transpose of the row vector $\begin{bmatrix}
\binom{j}{0} &
\cdots & 
\binom{j}{i} &
\cdots &
\binom{j}{j}
\end{bmatrix}$,
\item $e_j$ is the $j$-th standard basis vector of $\ZZ^{t+1}$
\end{itemize}
\end{defn}

\begin{lem}\label{HC}
With the notation  of Definition \ref{def:blocks}, the following statements hold true:
\begin{enumerate}
\item $\det C(a,b)_t =a^t-b^t$, if $t \ge 2$. 
\item The ideal of maximal minors of $H(a,b)_t$ is $I_t(H(a,b)_t)=(a,b)^{t}$.
\item If $(a,b)$ is an ideal of height two, then the minimal free resolution of  $R/(a,b)^{t}$ is $$0\to R^{t}\stackrel{H(a,b)_t}{\longrightarrow } R^{t+1} \to R \to R/(a,b)^{t}\to 0.$$
\end{enumerate}
\end{lem}

\begin{proof}
Applying Laplace expansion, it is easy to see that $\det C(a,b)_t=a^t+(-1)^{t+1}(-b)^t$ and that the  maximal minors of $H(a,b)_t$ generate $(a,b)^{t}$. Part $($c$)$ follows from $($b$)$ by the Hilbert-Burch theorem.
\end{proof}

We need another preparatory observation. 

\begin{lem}\label{3mat}
For any integer $n>0$, set $f=y^n-z^n, g=z^n-x^n, h=x^n-y^n \in R=\field[x,y,z]$. Fix an integer $ t > 0$ and consider the matrices of $\mathcal{M}_{(t+1)\times (t+1)}( R )$ given below, whose leftmost $t$ columns form $H(f,g)_{t}$. Then one has the determinantal formulas:

 \begin{enumerate}
 \item $\det \begin{bmatrix} 
 H (f,g)_t  & e_j
 \end{bmatrix} = (-1)^t f^{t-j+1} g^{j-1}$, for $1 \le j \le t+1$. \\[-.5em]
 
 \item $\det \begin{bmatrix} 
  & 0 \\
 H (f,g)_t  & \\
  & E_j
 \end{bmatrix} = (-1)^{t-j} g^{t-j} h^j$, for $0 \le j \le t$.  \\  \\

  \item $\det \begin{bmatrix} 
  & E_j \\
 H (f,g)_t  & \\
  & 0
 \end{bmatrix} = (-1)^{t-j}  f^{t-j} h^j$, for $0 \le j \le t$.  
\end{enumerate}
\end{lem}

\begin{proof}
All statements follow by expanding along the last column. For statements (2) and (3), one  uses part (1), the binomial formula and the identity $f+g=-h$.
\end{proof}

In the following we provide an explicit description of a set of minimal generators as well as the Betti numbers of the symbolic powers $I^{(nk)}$, where $I=(x(y^n-z^n),y(z^n-x^n),z(x^n-y^n))$  is the ideal of a Fermat configuration and  $n\geq 3$,  $k\geq 1$ are arbitrary integers. Our proof works inductively.  We begin by establishing the initial cases.

\begin{lem}
\label{lem:matrix X3}
Fix integers $n\geq3$ and $k \ge 1$ and set $f=y^n-z^n, g=z^n-x^n, h=x^n-y^n \in R=\field[x,y,z]$. Consider the block matrix $\mathbf{X}_{3} \in \mathcal{M}_{(k(n-3)+3n+1)\times(k(n-3)+3n)}( R )$ given by
$$
\mathbf{X}_{3}=\left[\begin{matrix}
H(f,g)_{k(n-3)} & U &  V & W \\ 
0 & C(x,y)_n & 0 & 0  \\ 
0 & 0& C(y,z)_n  & 0\\ 
0 & 0  & 0 & C(z,x)_n  \\ 
 \end{matrix}\right],$$
where all entries in columns 2 to $n$ of the $(k(n-3) +1)\times n$  matrices $U, V$ and $W$ are zero and the first columns of  $U, V$ and $W$ are defined as follows: 
\begin{itemize}

\item The first column of $U$ is $(-1)^{k (n-3)} x f  e_{n-2}$. 

\item The bottom $n-2$ entries of the first column of $V$ form the vector $(-1)^{(k-1)(n-3)} y g  E_{n-3}$, all other entries in this column are zero. 

\item The top  $(k-1)(n-3) +1$ entries of the first column of $W$ form the vector $(-1)^{n-3}z h  E_{(k-1) (n-3)}$, all other entries in this column are zero. 

\end{itemize}
 Then the following statements hold true:
\begin{enumerate}
\item The ideal of maximal minors of $\mathbf{X}_{3}$ is 
\begin{equation*}
        \begin{split}
I(\mathbf{X}_{3}) =  & (fgh)(f,g)^{k(n-3)} 
 + f^{(k-1)(n-3) +2} g^{n-2}x(x,y)^{n-1} \\
 & + g^{(k-1)(n-3)+2} h^{n-2}y(y,z)^{n-1}  +f^{n-2}h^{(k-1)(n-3)+2} z(z,x)^{n-1}.
\end{split}
\end{equation*}

\item  The minimal free resolution of  the cyclic module defined by the ideal above is 
\renewcommand\arraystretch{1}
 $$0\to
  \begin{array}{c}
R(-n [k(n-3) +4])^{k(n-3)}\\
\oplus\\
R(-n [k(n-3) +4]- 1)^{3n} \\
\end{array}
\stackrel{\mathbf{X}_{3}}{\longrightarrow } 
  \begin{array}{c}
R(-n [k(n-3) +3])^{k(n-3)+1}\\
\oplus\\
R(-n [k(n-3) +4])^{3n} \\
\end{array}
 \to R \to R/I(\mathbf{X}_{3})\to 0.
 $$
 
 \item $$I(\mathbf{X}_{3}) =  (f,g)^{k(n-3)+3} \cap (x, y)^n \cap (y,z)^n \cap (x,z)^n.$$
 
\end{enumerate}
\end{lem}

\begin{rem}
   \label{rem:n-th symb power}
(i) For $n=3$, we interpret $H(f,g)_{k(n-3)}$ as an empty matrix. So in this case the first column of $U$ is part of the first column of $\mathbf{X}_{3}$. 

(ii) If $k=1$, then Equation \eqref{eq:decomp I} implies that 
\[
I(\mathbf{X}_{3}) = (fgh)(f,g)^{n-3}+f^2 g^{n-2} x(x,y)^{n-1}+ g^2 h^{n-2} y(y,z)^{n-1}+f^{n-2}h^2z(z,x)^{n-1}
\] 
is the $n$-th symbolic power of $I=(x(y^n-z^n),y(z^n-x^n),z(x^n-y^n))$. 
\end{rem}

\begin{proof}[Proof of Lemma \ref{lem:matrix X3}]
$(1)$ We start by examining the maximal minors of  $\mathbf{X}_{3}$ resulting from discarding one of the first $k(n-3) + 1$ rows. By properties of block upper-triangular matrices, such a minor is the product of four determinants: the minor of $H(f,g)_{n-3}$ corresponding to the deleted row, $\det(C(x,y)_n)$, $\det(C(y,z)_n)$ and $\det(C(z,x)_n)$. Using the formulas in Lemma \ref{HC}, it is clear that these minors generate the ideal $(fgh)(f,g)^{(n-3)}$.

To analyze the maximal minors of  $\mathbf{X}_{3}$ resulting from discarding one of the next $n$ rows note that deleting one row of $C(x,y)_n$ leaves a block upper-triangular matrix with three diagonal blocks consisting of: the first $k(n-3) + n$ rows and columns (corresponding to the blocks $H(f,g)_{n-3}$, $C(x,y)_n$) and the blocks  $C(y,z)_n$ and $C(z,x)_n$. The determinant of the latter two blocks are $f,g$, while  for the first block one gets the product of a minor of $H(x,y)_{n-1}$ and the determinant of the matrix formed by $H(f,g)_{k(n-3)}$ and the first column of $U$. This latter determinant is $f^{(k-1)(n-3) +1} g^{n-3}x$ by Lemma \ref{3mat}. Hence, Lemma \ref{HC} shows that these minors generate the ideal $ f^{(k-1)(n-3) +2} g^{n-2}x(x,y)^{n-1}$. 

For analyzing the maximal minors of  $\mathbf{X}_{3}$ resulting from discarding one of the next $n$ rows corresponding to the $C(y,z)_n$ block, we permute rows and columns of $\mathbf{X}_{3}$ to obtain a matrix 
$$
\mathbf{X}_{3}' = 
\left[\begin{matrix}
H(f,g)_{k(n-3)} & V &  U & W \\ 
0 & C(y,z)_n & 0 & 0  \\ 
0 & 0& C(x,y)_n  & 0\\ 
0 & 0  & 0 & C(z,x)_n  \\ 
 \end{matrix}\right].
 $$ 
Thus to find  the maximal minors of  $\mathbf{X}_{3}$ resulting from discarding one of the  rows corresponding to the $C(y,z)_n$ block, it suffices to analyze the corresponding minors of $\mathbf{X}_{3}'$ above. Arguing as in the case of deleting a row of $\mathbf{X}_{3}$ corresponding to the $C(x,y)_n$ block, we see that the maximal minors of  $\mathbf{X}_{3}$ resulting from discarding one of the  rows corresponding to the $C(y,z)_n$ block generate the ideal $g^{(k-1)(n-3)+2} h^{n-2}y(y,z)^{n-1}$.


A similar argument  yields that the minors corresponding to deleting one of the last $n$ rows of  $\mathbf{X}_{3}$ generate the ideal $f^{n-2}h^{(k-1)(n-3)+2} z(z,x)^{n-1}$.  Details are left to the reader.

$(2)$ By $(1)$, the ideal $I (\mathbf{X}_{3})$ contains the polynomials $f^{k(n-3)+1} gh$ and  $f^{(k-1)(n-3) +2} g^{n-2} x^n+g^{(k-1)(n-3) +2} h^{n-2} y^n+f^{n-2}h^{(k-1)(n-3) +2} z^n$. Since none of the (linear) divisors of $f, g$, and $h$ divides the latter polynomial, the two stated polynomials form a regular sequence of length two inside $I (\mathbf{X}_{3})$. Hence, an application of the Hilbert-Burch theorem gives the stated minimal resolution.

$(3)$ Set 
\[
J =  (f,g)^{k(n-3)+3} \cap (x, y)^n \cap (y,z)^n \cap (x,z)^n. 
\]
Note that $f \in (y,z)^n$, \ $g \in (x,z)^n$, and $h \in (x,y)^n$. Thus, using the set of generators of $I (\mathbf{X}_{3})$ given in $(1)$ one sees that $I (\mathbf{X}_{3}) \subseteq J$. In order to establish equality, it is sufficient to show that the ideals on both sides  are unmixed  and have  the same multiplicity.  The unmixedness of $J$ follows from its definition. The ideal $I (\mathbf{X}_{3})$ is unmixed as well because $R/I (\mathbf{X}_{3})$ is Cohen-Macaulay by $(2)$. 

It remains to compare the multiplicities. By \cite[Theorem 4.2 (2)]{refEis2}, we may compute the multiplicity of $R/I (\mathbf{X}_{3})$ as 
\begin{eqnarray*}
e( R/I (\mathbf{X}_{3}) & = & H_{R/I (\mathbf{X}_{3})} (\reg ( R/I (\mathbf{X}_{3})+{\rm pd}( R/I (\mathbf{X}_{3}) -2)\\
& = &  H_{R/I (\mathbf{X}_{3})} (n [k (n-3) + 4] -1), 
\end{eqnarray*}
where $H_M (j) = \dim_K [M]_j$ denotes the Hilbert function of a graded module $M$ in degree $j$ and we used the resolution given in $(2)$ to compute the regularity of $R/I (\mathbf{X}_{3})$. Taking this resolution into account again, the above  formula can be evaluated as follows:
\begin{eqnarray*}
e( R/I (\mathbf{X}_{3}) )&=&H_{R/I (\mathbf{X}_{3})} (n [k (n-3) + 4] -1)\\
&=&H_R(n [k (n-3) + 4] -1)- [k (n-3) + 1] \cdot  H_R(n-1)\\
&=& n^2 \binom{k (n-3) + 4}{2} +3 \binom{n+1}{2}.  
\end{eqnarray*}

We now determine the multiplicity of $R/J$. By the linearity formula, where $\p_i$ are the ideals of the $n^2$ points of the scheme defined by $(f, g)$, one has
\[
e(R/(f, g)^{k(n-3)+3})=\sum_{i=1}^{n^{2}} e(R/\p_i)e(R_{\p_i}/\p_{k(n-3)+3}^{n}R_{\p_i})= n^2 \binom{k(n-3)+4}{2}.
\]
It follows that 
\[
e (R/J) = n^2 \binom{k(n-3)+4}{2} + 3 \binom{n+1}{2}. 
\]
We conclude that $e (R/J) = e( R/I (\mathbf{X}_{3}))$, and thus $I (\mathbf{X}_{3}) = J$, as desired. 
\end{proof}

Now we extend the above results to higher symbolic powers. 

\begin{thm}
    \label{thm:gens symb powers}
Let $n \ge 3$ and $k \ge 1$ be integers  and consider the ideal $I = (x f, y g, z h)$ of the Fermat configuration, where  $f=y^n-z^n, g=z^n-x^n, h=x^n-y^n \in R=\field[x,y,z]$. Then the $k n$-th symbolic power of $I$ has the following set of minimal generators 
\begin{equation*}
\begin{split}
I^{(k n)} = & (fgh)^k \cdot (f,g)^{(n-3)k} \\
&+ \sum_{i=1}^{k} f^{(k-i)(n-2) + 2i} g^{k + i (n-3)} h^{k-i} x^{(i-1)n+1} \cdot (x,y)^{n-1}\\
 &+ \sum_{i=1}^{k} f^{k-i} g^{(k-i)(n-2) + 2i} h^{k + i (n-3)}   y^{(i-1)n+1} \cdot (y,z)^{n-1}\\
 & +  \sum_{i=1}^{k} f^{k + i (n-3)} g^{k-i} h^{(k-i)(n-2) + 2i}  z^{(i-1)n+1} \cdot (z,x)^{n-1}. 
\end{split}
\end{equation*}
\end{thm} 

This is a consequence of the following more general result, which also describes the Hilbert-Burch matrix of $I^{(k n)}$ and other related ideals. 

\begin{thm} 
     \label{thm:recursion}
Fix integers $n\geq3$ and $k\geq 1$, put $f=y^n-z^n, g=z^n-x^n, h=x^n-y^n \in  R=\field[x,y,z]$, and  
define  recursively   block matrices $\mathbf{X}_{j}\in \mathcal{M}_ {(k(n-3)+j n + 1)\times(k(n-3)+j n)}(R)$, for $0 \le j \le 3 k$, as follows: 

\noindent If $j \ge 1$  write $j = 3 i + r$ with integers $i, r$ such that $0 \le i, \ 1 \le r \le 3$, put 
$\mathbf{X}_{0} = H(f, g)_{k (n-3)}$ and 
\[
\mathbf{X}_{j} = \begin{bmatrix}
\mathbf{X}_{j-1} & Y_j \\
0 & Z_j
\end{bmatrix}, 
\]
where 
\[
Z_j = \begin{cases}
C(x,y)_n & \text{if } r = 1 \\
C(y,z)_n & \text{if } r = 2 \\
C(z,x)_n & \text{if } r = 3
\end{cases}
\quad
\text{ and }
\quad
Y_j = \begin{bmatrix}
S_j & 0 \\
0 & 0
\end{bmatrix}
\]
with  matrix $S_j\in \mathcal{M}_{[k (n-3) + 1 + (i-1)n] \times 1}(R)$ such that 
\[
S_1 = (-1)^{k(n-3)}x f e_{n-2}, \quad S_2 = 
\begin{bmatrix}
0 \\
(-1)^{(k-1)(n-3)} y g E_{n-3}
\end{bmatrix},  \quad S_3 = 
\begin{bmatrix}
(-1)^{n-3} z h E_{k (n-3)} \\
0
\end{bmatrix}
\] 
and,  if $4 \le j \le 3k$, 
\[
\det \begin{bmatrix}
\mathbf{X}_{j-3} & S_j
\end{bmatrix}
 = \begin{cases}
f^{(k-1-i) ( n-3) + 2i} g^{(i+1) (n-2) - 2}  x^{in +1}  & \text{if } r = 1 \\
g^{(k-1-i) ( n-3) + 2i} h^{(i+1) (n-2) - 1}  y^{in +1}  & \text{if } r = 2 \\
f^{(i+1) (n-2) - 1} h^{(k-1-i) ( n-3) + 2i + 1}  z^{in +1}  & \text{if } r = 3. 
\end{cases}
\]
Such column vectors $S_j$ do exist. 

Then the ideal of maximal minors of $\mathbf{X}_{j}$ has the following properties: 
\begin{enumerate}

\item If $1 \le j \le 3k$, then 
\[
I(\mathbf{X}_{j})  = \begin{cases}
h \cdot I (\mathbf{X}_{j-1}) +f^{(k-1-i) ( n-3) + 2i + 1} g^{(i+1) (n-2) - 1}  x^{in +1} \cdot (x, y)^{n-1} & \text{if } r = 1 \\
f \cdot I (\mathbf{X}_{j-1}) +   g^{(k-1-i) ( n-3) + 2i + 1} h^{(i+1) (n-2)}  y^{in +1}  \cdot (y, z)^{n-1} & \text{if } r = 2 \\
g \cdot I (\mathbf{X}_{j-1}) +  f^{(i+1) (n-2)} h^{(k-1-i) ( n-3) + 2i + 2}  z^{in +1} \cdot (x, z)^{n-1}& \text{if } r = 3.  
\end{cases}
\]

\item A minimal free resolution of $I(\mathbf{X}_{j})$ is 
\renewcommand\arraystretch{1}
 \[
 0\to
  \begin{array}{c}
R(-n [k(n-3) + j + 1 ])^{k(n-3)}\\
\oplus \\
\bigoplus_{\ell=1}^i R(-n [k(n-3) +j+\ell] - 1)^{3n}  \\
\oplus\\
R(-n [k(n-3) +j + i + 1]- 1)^{rn} \\
\end{array}
\stackrel{\mathbf{X}_{j}}{\longrightarrow } 
  \begin{array}{c}
R(-n [k(n-3) +j])^{k(n-3)+1}\\
\oplus \\
\bigoplus_{\ell=1}^i R(-n [k(n-3) +j+\ell])^{3n}  \\
\oplus\\
R(-n [k(n-3) +j+i+1])^{rn} \\
\end{array}
 \to I(\mathbf{X}_{j})\to 0.
 \]

\item If $1 \le j \le 3k$, then 
\[
I( \mathbf{X}_{j})  = \begin{cases}
(f, g)^{k (n-3) + j} \cap  (x, y)^{(i+1) n} \cap (y,z)^{i n} \cap (x, z)^{i n} & \text{if } r = 1 \\
(f, g)^{k (n-3) + j} \cap  (x, y)^{(i+1) n} \cap (y,z)^{(i+1) n} \cap (x, z)^{i n} & \text{if } r = 2 \\
(f, g)^{k (n-3) + j} \cap   (x, y)^{(i+1) n} \cap (y,z)^{(i+1) n} \cap (x, z)^{(i+1) n}& \text{if } r = 3. 
\end{cases}
\] 

\end{enumerate}
\end{thm} 

\begin{rem}
(i) The matrix $\mathbf{X}_{3}$ in the above theorem is the same as the matrix $\mathbf{X}_{3}$ given in Lemma \ref{lem:matrix X3}. 

(ii) If $n =3$, then $\mathbf{X}_{0}$ is an empty matrix, and thus $\mathbf{X}_{1} = \begin{bmatrix}
Y_1 \\
Z_1
\end{bmatrix}. $
\end{rem}

\begin{proof}[Proof of Theorem \ref{thm:recursion}] 

If $j = 3$, then claims $(2)$ and $(3)$ have been shown in Lemma \ref{lem:matrix X3}. Furthermore, there the minimal generators of $I( \mathbf{X}_{3})$ are given.  Arguments entirely similar to those in the proof of Lemma \ref{lem:matrix X3} establish the analogous statements for $I( \mathbf{X}_{2})$ and $I( \mathbf{X}_{1})$. From the generating sets of these ideals one infers that claim $(1)$ is true if $1 \le j \le 3$. 

Let $j \ge 4$, and thus $i \ge 1$.  We show all assertions simultaneously assuming their correctness for smaller matrices. 

$(0)$ We begin by proving that a column vector $S_j$ with the claimed property exists. We check this depending on the remainder $r$. 

Let $r = 1$, so $j = 3 i + 1$. Recall that $f \in (y,z)^n$, \ $g \in (x,z)^n$, and $h \in (x,y)^n$. It follows that 
\begin{equation*}
\begin{split}
f^{(k-1-i) ( n-3) + 2i} g^{(i+1) (n-2) - 2}  x^{in +1} \in & \ (f, g)^{k (n-3) + 3i-2} \cap  (x, y)^{i n} \cap (y,z)^{(i-1)n} \cap (x, z)^{(i-1)n} \\
& = \ I(\mathbf{X}_{3 (i-1) + 1}) = I(\mathbf{X}_{j-3}), 
\end{split}
\end{equation*}
where the first equality is due to the induction hypothesis. Hence $f^{(k-1-i) ( n-3) + 2i} g^{(i+1) (n-2) - 2}  x^{in +1}$ is a linear combination of the minimal generators of $I(\XX_{j-3})$. These generators can be taken as the maximal minors of $\XX_{j-3}$. Thus, collecting the coefficients of the minors  with suitable signs in a column vector gives the desired vector $S_j$. 

Let $r =2$. Then the induction hypothesis implies  
\begin{equation*}
\begin{split}
g^{(k-1-i) ( n-3) + 2i} h^{(i+1) (n-2) - 1}  y^{in +1} \in & \ (f, g)^{k (n-3) + 3i-1} \cap  (x, y)^{i n} \cap (y,z)^{i n} \cap (x, z)^{(i-1)n} \\
& = \ I(\mathbf{X}_{3 (i-1) + 2}) = I(\mathbf{X}_{j-3}). 
\end{split}
\end{equation*}
Now the existence of a vector $S_j$ follows as in the case where $r=1$. 

If $r = 3$, one similarly gets 
\begin{equation*}
\begin{split}
f^{(i+1) (n-2) - 1} h^{(k-1-i) ( n-3) + 2i + 1}  z^{in +1} \in & \ (f, g)^{k (n-3) + 3i-1} \cap  (x, y)^{i n} \cap (y,z)^{i n} \cap (x, z)^{i n} \\
& = \ I(\mathbf{X}_{3 (i-1) + 3}) = I(\mathbf{X}_{j-3}),  
\end{split}
\end{equation*}
and the existence of $S_j$ follows. 


Next we provide the arguments necessary to justify claims (1)--(3).

$(1)$ Recall that $
\mathbf{X}_{j} = \begin{bmatrix}
\mathbf{X}_{j-1} & Y_j \\
0 & Z_j
\end{bmatrix}
$. We start by examining the maximal minors of  $\mathbf{X}_{j}$ resulting from discarding one of the rows in which the block $\mathbf{X}_{j-1}$ is found. By properties of block upper-triangular matrices, such a minor is the product of a maximal minor of $\mathbf{X}_{j-1}$ and $\det(Z_j)$. Therefore, these minors generate 
$$\det(Z_j)I(\mathbf{X}_{j-1})=
\begin{cases}
h\cdot I(\mathbf{X}_{j-1}) &\text{if } r=1\\
f\cdot I(\mathbf{X}_{j-1}) & \text{if } r=2\\
g\cdot I(\mathbf{X}_{j-1}) &\text{if } r=3.
\end{cases}$$
Analyzing the maximal minors of  $\mathbf{X}_{j}$ resulting from discarding one of the rows corresponding to the lower blocks, one gets the product of a minor of $H(x,y)_{n-1}, H(y,z)_{n-1}$ or $H(x,z)_{n-1}$ (depending on $r$) and the determinant of the matrix formed by $\mathbf{X}_{j-1}$ and the first column of $Y_j$, i.e. $\det \left[ \begin{matrix}\mathbf{X}_{j-1} & S_j\end{matrix}\right]$. The ideals generated by the former minors are given in Lemma \ref{HC} and the value for this latter determinant is given by hypothesis. Hence, these last minors of $\mathbf{X}_{j}$ generate the ideal 
$$
\begin{cases}
f^{(k-1-i) ( n-3) + 2i} g^{(i+1) (n-2) - 2}  x^{in +1} (x,y)^{n-1} & \text{if } r = 1 \\
g^{(k-1-i) ( n-3) + 2i} h^{(i+1) (n-2) - 1}  y^{in +1} (y,z)^{n-1} & \text{if } r = 2 \\
f^{(i+1) (n-2) - 1} h^{(k-1-i) ( n-3) + 2i + 1}  z^{in +1} (x,z)^{n-1} & \text{if } r = 3. 
\end{cases}$$
Summing the two ideals above gives the formulas in part (1)

$(2)$  By the inductive hypothesis $I (\mathbf{X}_{j-1})$ is a perfect height two ideal, therefore it is not contained in the union of the prime ideals generated by each of the linear divisors of $f, g, h$ and the linear forms $x,y,z$. Consequently there is a polynomial $\alpha\in I (\mathbf{X}_{j-1})$ that is not divisible by any of the linear factors of $f,g$, nor by $x$.  If $r=1$, consider the polynomial $h\alpha$, which is by $(1)$ an element of  $I (\mathbf{X}_{j})$. We shall find a polynomial $\beta\in (x,y)^{n-1}$ so that $h\alpha$ and $ f^{(k-1-i) ( n-3) + 2i + 1} g^{(i+1) (n-2) - 1}  x^{in +1}\beta$ form a regular sequence in  $I (\mathbf{X}_{j})$. Indeed, one can pick $\beta\in (x,y)^{n-1}$ so that $h\alpha$ and $\beta$ form a regular sequence. This insures that the forms $h\alpha$ and $ f^{(k-1-i) ( n-3) + 2i + 1} g^{(i+1) (n-2) - 1}  x^{in +1}\beta$ have no common factors of positive degree, thus they form a regular sequence. Analogous arguments show that the grade of $I (\mathbf{X}_{j})$ is 2 in the remaining cases $r=2$ and $r=3$. 

The claim on the minimal free resolution of $I (\mathbf{X}_{j})$ now follows by Hilbert-Burch. The formulas for the graded shifts in the resolution are found by taking into account the inductive hypothesis, together with the formulas for generators of $I (\mathbf{X}_{j})$ found in part (1) and the structure of the blocks of the matrix $\mathbf{X}_{j}$, specifically the fact that the entries of $Z_j$ are linear.

$(3)$ Set 
\[
J(n,j) =  \begin{cases}
(f, g)^{k (n-3) + j} \cap  (x, y)^{(i+1) n} \cap (y,z)^{in} \cap (x, z)^{in} & \text{if } r = 1 \\
(f, g)^{k (n-3) + j} \cap  (x, y)^{(i+1) n} \cap (y,z)^{(i+1) n} \cap (x, z)^{in} & \text{if } r = 2 \\
(f, g)^{k (n-3) + j} \cap   (x, y)^{(i+1) n} \cap (y,z)^{(i+1) n} \cap (x, z)^{(i+1) n}& \text{if } r = 3. 
\end{cases}
\]
Using the recursive formula for  $I (\mathbf{X}_{j})$ given in $(1)$ and the inductive hypothesis $I (\mathbf{X}_{j-1})=J(n,j-1)$, one sees that $I (\mathbf{X}_{j}) \subseteq J(n,j)$.  In order to establish equality $I (\mathbf{X}_{j})=J(n,j)$ it is sufficient to show that the ideals on both sides  are unmixed  and have  the same multiplicity. The unmixedness of $J(n,j)$ follows from its definition. The ideal $I (\mathbf{X}_{3})$ is unmixed as well because $R/I (\mathbf{X}_{j})$ is Cohen-Macaulay by $(2)$. 
It remains to compare the multiplicities. 
 Using \cite[Theorem 4.2 (2)]{refEis2} and the resolution in $(2)$ we compute
\begin{eqnarray*}
e( R/I(\mathbf{X}_{j})) &=& H_{R/I(\mathbf{X}_{j})}({\rm reg}( R/I(\mathbf{X}_{j})+{\rm pd}( R/I(\mathbf{X}_{j}))-2)\\
&=& H_{R/I(\mathbf{X}_{j})}(n[k(n-3)+j +i+1]-1) \\
&=& H_R(n[k(n-3)+j +i+1]-1)-(k(n-3)+1)H_R(n(i+1)-1)\\
&&-3n\sum_{\ell=1}^{i}H_R(n(i+1-\ell)-1)+k(n-3)H_R(ni-1) +3n\sum_{\ell=1}^{i}H_R(n(i+1-\ell)-2)\\
&=& \binom{n[k(n-3)+j +i+1]+1}{2}-(k(n-3)+1)\binom{n(i+1)+1}{2}\\
&&+k(n-3)\binom{ni+1}{2}-3n^2\binom{i+1}{2},
\end{eqnarray*}
where some of the terms in the above formula are obtained by evaluating
\begin{eqnarray*}
\sum_{\ell=1}^{i}H_R(n(i+1-\ell)-1)-\sum_{\ell=1}^{i}H_R(n(i+1-\ell)-2)&=&\sum_{\ell=1}^{i} (n(i+1-\ell))=\frac{ni(i+1)}{2}.\\
\end{eqnarray*}

It can be verified by straightforward computation that 
$$e( R/I(\mathbf{X}_{j}))=e( R/J(n,k))=
\begin{cases}
n^2\binom{k(n-3)+j +1}{2}+2\binom{in+1}{2}+\binom{(i+1)n+1}{2}& \text{if } r = 1 \\
n^2\binom{k(n-3)+j +1}{2}+\binom{in+1}{2}+2\binom{(i+1)n+1}{2}& \text{if } r = 2 \\
n^2\binom{k(n-3)+j +1}{2}+3\binom{(i+1)n+1}{2}& \text{if } r = 3, \\
\end{cases}
$$
whence $I(\mathbf{X}_{j})=J(n,k)$ follows.

 \end{proof} 
 
 \begin{proof}[Proof of Theorem \ref{thm:gens symb powers}]
 We use the notation of Theorem \ref{thm:recursion}. Its part (3) shows that $I(\XX_{3k}) = I^{(kn)}$.  Using the recursion given in Theorem \ref{thm:recursion}(1), a routine computation yields the claimed generating set of $I^{(k n)}$. It is minimal because it consists of $k n +1$ polynomials, which is the number of minimal generators of $I^{(k n)}$ by Theorem \ref{thm:recursion}(2). 
 \end{proof}
 
 \begin{rem}\label{rem:compact gens symb powers} The conclusion of Theorem \ref{thm:gens symb powers} can be rewritten more compactly by presenting  $I^{(k n)}$ as a sum of four ideals: 
\begin{eqnarray*}
I^{(k n)} &= &(fgh)^k(f,g)^{(n-3)k} \\
&& + x (x,y)^{n-1} g^{n-2}f^2 \cdot (f^{n-2} g h, \ g^{n-2}f^2 x^n)^{k-1}\\[3pt]
&& + y (y,z)^{n-1} h^{n-2}g^2 \cdot (f g^{n-2} h, \ h^{n-2}g^2 y^n)^{k-1}\\[3pt] 
&& + z (z,x)^{n-1} f^{n-2}h^2 \cdot (g f h^{n-2}, \ f^{n-2}h^2z^n)^{k-1}. 
%
\end{eqnarray*}
\end{rem}

\subsection{Regularity of symbolic powers}

In Corollary \ref{cor:res-ordinary} we gave a formula for the regularity of ordinary powers of Fermat ideals, which is a linear function in $r$ for all $r\geq 2$: $\reg(I^r)=r(n+1)+n-1$. In fact it is known by \cite{refCHT} that  $\reg(I^{r})$ becomes a linear function of $r$ for large enough values of the exponent.  We now turn our attention towards the Castelnuovo-Mumford regularity of the symbolic powers. In the case of the Fermat ideals, it turns out that this is also given by a linear function for high enough powers, as we will show in Theorem \ref{thm:reg-symb}.  By contrast, in general it can only be shown as in \cite[Theorem 4.3]{refCHT}  that, if $\srees(I)$ is finitely generated, then $\reg(I^{(m)})$ is a periodic linear function for $m$ large enough, i.e. there exist integers $a_i$ and $b_i$ such that $\reg(I^{(m)})=a_im+b_i$ for $t\equiv i$ mod $n$ and $t\gg0$.  

We now proceed to give an explicit formula for the regularity of high enough symbolic powers of Fermat ideals.

\begin{thm}\label{thm:reg-symb}
Let $I=(x(y^n-z^n),y(z^n-x^n),z(x^n-y^n))$ with $n\geq 3$. The symbolic powers of $I$ have their Castelnuovo-Mumford regularity given by
$$\reg(I^{(m)})=m(n+1), \text{ for } m\gg0.$$
\end{thm}
\begin{proof}
We begin by proving that the conclusion holds for $m=n$ and $m=n-1$. From part $(2)$ of Lemma \ref{lem:matrix X3} (with $k=1$), we have that $$\reg(I^{(n)})=n(n+1).$$ More generally, it follows by part $(2)$ of Theorem \ref{thm:recursion} that $\reg(I^{(nk)}=\reg(I(\mathbf{X}_{3k}))=nk(n+1)$ for all integers $k\geq 1$. Next we  set $f=y^n-z^n, g=z^n-x^n, h=x^n-y^n$ and we consider the block matrix $\mathbf{X}_{3}' \in \mathcal{M}_{(k(n-3)+3n+1)\times(k(n-3)+3n)}( R )$ given by
$$
\mathbf{X}_{3}'=\left[\begin{matrix}
H(f,g)_{n-4} & U' &  V' & W' \\ 
0 & C(x,y)_n & 0 & 0  \\ 
0 & 0& C(y,z)_n  & 0\\ 
0 & 0  & 0 & C(z,x)_n  \\ 
 \end{matrix}\right],$$
where the matrices $U', V', W'$ are defined analogously to the ones in Lemma \ref{lem:matrix X3}:
\begin{itemize}

\item The first column of $U'$ is $(-1)^{n-4}f e_{n-3}$, all other entries are zero. 

\item The first column of $V'$ is the vector $g  E_{n-4}$, all other entries  are zero. 

\item The first column of $W'$ is the vector $h  E_{n-4}$, all other entries are zero. 

\end{itemize}
 We make the following claims if $n \ge 4$:
\begin{enumerate}
\item The ideal of maximal minors of $\mathbf{X}_{3}'$ is 
\begin{equation*}
I(\mathbf{X}_{3}') =   (fgh)(f,g)^{n-4} 
 + f^2 g^{n-3}(x,y)^{n-1}  + g^{2} h^{n-3}(y,z)^{n-1}  +f^{n-3}h^{2}(z,x)^{n-1}.
\end{equation*}

\item  The minimal free resolution of  the cyclic module defined by the ideal above is 
\renewcommand\arraystretch{1}
 $$0\to
  \begin{array}{c}
R(-n^2)^{4n-4} \\
\end{array}
\stackrel{\mathbf{X}_{3}'}{\longrightarrow } 
  \begin{array}{c}
R(-n^2+n)^{n-3}\\
\oplus\\
R(-n^2+1)^{3n} \\
\end{array}
 \to R \to R/I(\mathbf{X}_{3}')\to 0.
 $$
 
 \item $$I(\mathbf{X}_{3}') =  (f,g)^{n-1} \cap (x, y)^{n-1} \cap (y,z)^{n-1} \cap (x,z)^{n-1}.$$
 
\end{enumerate}
The three claims follow exactly like in the proof of Lemma \ref{lem:matrix X3}. We leave the details to the diligent reader. Based on the free resolution given by our claim $(2)$ we deduce that $$\reg(I^{(n-1)})=n^2-1=(n-1)(n+1).$$
One checks that this equality is also true if $n = 3$.

Consider the set $\mathcal{S}=\{an+b(n-1) \ | \ a,b \in \mathbb{N}\}$. We will prove that for any $m\in \mathcal{S}$, we have $\reg(I^{(m)})=m(n+1)$. Indeed, set $m=an+b(n-1)$ and notice the containments
$$I^m=I^{an}I^{b(n-1)}\subseteq \left(I^{(n)}\right)^a\left(I^{(n-1)}\right)^b\subseteq I^{(m)},$$ which yield that $I^{(m)}=\left( \left(I^{(n)}\right)^a\left(I^{(n-1)}\right)^b\right)^{\rm{sat}}$, where the superscript sat denotes saturation with respect to the homogeneous maximal ideal. Consequently, the cohomological characterization of the Castelnuovo-Mumford regularity implies the inequality
\[
\reg\left(\left(I^{(n)}\right)^a\left(I^{(n-1)}\right)^b\right)\geq \reg(I^{(m)}).
\] 

Furthermore, iterated applications of \cite[Theorem 2.5]{refCH}, using the fact that $\dim(R/I^{(n)})=\dim(R/I^{(n-1)})=1$, yield that 
$$\reg\left(\left(I^{(n)}\right)^a\left(I^{(n-1)}\right)^b\right) \leq a\reg(I^{(n)})+b\reg(I^{(n-1)}).$$ 

Putting everything together gives
$$\reg(I^{(m)})\leq \reg((I^{(n)})^a(I^{(n-1)})^b) \leq a\reg(I^{(n)})+b\reg(I^{(n-1)})=an(n+1)+b(n-1)(n+1)=m(n+1).$$

To establish the opposite inequality it is sufficient to prove that there exist minimal generators of $I^{(m)}$ of degree at least $m(n+1)$. Towards this end we show that, if $\tau\in I^{(m)}$ and $\deg(\tau)<m(n+1)$, then $\tau\in(fgh)$. This follows easily by Bezout's Theorem. Indeed, consider any linear factor $\ell$ of the product $fgh$. Since the line defined by $\ell$ contains $n+1$ points at which $\tau$ vanishes to order at least $m$, the intersection multiplicity of $\tau$ and $\ell$ is at least $(n+1)m>\deg(\tau)\deg(\ell)$. Thus $\ell\mid \tau$ for every such linear form $\ell$, whence $(fgh)\mid \tau$. This shows that the generators of $I^{(m)}$ of degrees less than $m(n+1)$ generate an ideal of height one properly contained in $I^{(m)}$, therefore there must be additional minimal generators of higher degree. This gives in particular that $\reg(I^{(m)})\geq m(n+1)$.

 The two inequalities above prove that $\reg(I^{(m)})= m(n+1)$ for $m\in\mathcal{S}$. Noting that every large enough positive integer is an element of the semigroup $\mathcal{S}$, since $\gcd(n,n-1)=1$, finishes the proof.
\end{proof}

\begin{rem} It is natural  to ask for effective bounds on the magnitude of $m$ that would insure the formula in Theorem \ref{thm:reg-symb} applies. The proof of Theorem \ref{thm:reg-symb} gives that the Frobenius number of the semigroup $\mathcal{S}$ is one such bound. By work of Sylvester \cite{refSy} this Frobenius number is $n(n-1)-n-(n-1)=n^2-3n+1$, thus we obtain $$\reg(I^{(m)})=m(n+1) \text{ for }  m\geq n^2-3n+2.$$ Computational evidence suggests  that in fact $\reg(I^{(m)})=m(n+1)$ for  $m\geq n-2$. Indeed, this is true if $n = 3$ by using also Corollary \ref{cor:res-ordinary}.
\end{rem}


\section{Symbolic Rees algebras of Fermat ideals are Noetherian} \label{sec:symb Rees} 


It is well-known that, unlike the ordinary Rees algebra, the symbolic Rees algebra of a homogeneous ideal may in general not be Noetherian, even for ideals defining reduced sets of points.  In this section we show that for the Fermat family of ideals the symbolic Rees algebras are in fact Noetherian. A particular case of this result (the case $n=3$) can be found in \cite[Proposition 1.1]{refHS}, where it is derived as a direct consequence of a result in  \cite{refHH}. Our methods here are entirely disjoint from the approach of \cite{refHH, refHS}. 

The key to our approach is the following result.

\begin{prop}\label{prop:Ink}
Let $I=(x(y^n-z^n),y(z^n-x^n),z(x^n-y^n))$, with $n\geq 3$. 
Then 
$$I^{(nk)}={I^{(n)}}^k \mbox{ for all integers }k\geq 1.$$
\end{prop}

\begin{proof}
Since the assertion is  tautologically true if $k = 1$, we assume now $k \ge 2$. We are going to establish the following claim: 

For each $k \ge 2$, 
\begin{equation}
     \label{eq:prod symb powers}
  I^{(k n)}  \subseteq  I^{(n)} \cdot I^{((k-1) n)}. 
\end{equation}

We check this using the list of minimal generators given in Theorem \ref{thm:gens symb powers}. It gives that $I^{(k n)}$ contains 
\[
(fgh)^k \cdot (f,g)^{(n-3)k} = [ (fgh) \cdot (f,g)^{n-3}] \cdot [(fgh)^{k-1} \cdot (f,g)^{(n-3)(k-1)}]. 
\]
Hence, $(fgh)^k \cdot (f,g)^{(n-3)k} \subset I^{(n)} \cdot I^{((k-1) n)}$. 

Next, we show that, for each $i \in [k]$, 
\begin{equation}
     \label{eq:inclusion gens}
f^{(k-i)(n-2) + 2i} g^{k + i (n-3)} h^{k-i} x^{(i-1)n+1} \cdot (x,y)^{n-1} \subset I^{(n)} \cdot I^{((k-1) n)}. 
\end{equation}
To this end rewrite the product on the left-hand side as 
\[
[ f^2 g^{n-2} x \cdot (x,y)^{n-1}] \cdot   f^{(k-i)(n-2) + 2i - 2} g^{k - 1 + (i-1) (n-3)} h^{k-i} x^{(i-1)n+1}. 
\]
Notice that $f^2 g^{n-2} x \cdot (x,y)^{n-1} \subset I^{(n)}$ (see, e.g., Remark \ref{rem:n-th symb power}(ii)). Moreover, we get 
\[
f^{(k-i)(n-2) + 2i - 2} g^{k - 1 + (i-1) (n-3)} h^{k-i} x^{(i-1)n+1} \in I^{((k-1) n)}
\]
because $h^{k-i} x^{(i-1)n} \in (x, y)^{(k-1) n}$, \  $f^{k-i} x^{(i-1)n} \in (y,z)^{(k-1) n}$, and $g^{k-1} \in (x,z)^{(k-1)n}$. Now the containment \eqref{eq:inclusion gens} follows. 

Similarly, one proves for each $i \in [k]$, 
\[
f^{k-i} g^{(k-i)(n-2) + 2i} h^{k + i (n-3)}   y^{(i-1)n+1} \cdot (y,z)^{n-1} \subset I^{(n)} \cdot I^{((k-1) n)} 
\]
and
\[
f^{k + i (n-3)} g^{k-i} h^{(k-i)(n-2) + 2i}  z^{(i-1)n+1} \cdot (z,x)^{n-1} \subset I^{(n)} \cdot I^{((k-1) n)}. 
\]
Comparing with Theorem \ref{thm:gens symb powers}, we have shown that each minimal generator of $I^{(kn)}$ is contained in $I^{(n)} \cdot I^{((k-1) n)}$, which gives the desired containment \eqref{eq:prod symb powers}. 
Since for every ideal $I$ one has the inclusion $I^{(n)} \cdot I^{(k-1)n}\subseteq I^{(kn)}$, we obtain  the equality  $I^{(n)} \cdot I^{(k-1)n}= I^{(kn)}$, which together with the inductive hypothesis finishes the proof.
\end{proof}

\begin{rem} 
   \label{rem:comp}
For $n=3$, the above Proposition was also proved in \cite[Proposition 1.1]{refHS} using a different method based on \cite[Proposition 3.5]{refHH}. We note that one cannot apply \cite[Proposition 3.5]{refHH} directly for proving this property  of Fermat ideals when the parameter $n$ is greater than 3. Indeed,  since, in the notation of \cite{refHH}, we have that the minimum degree of an element of a minimal set of generators for $I^{(n)}$ is $\alpha_n=\alpha(I^{(n)})=n^2$ and the maximum degree of an element of a minimal set of generators for $I^{(n)}$ is   $\beta_n=\beta(I^{(n)})=n^2+n$ we obtain $\alpha_n\beta_n=n^2(n^2+n)$. The hypothesis needed to employ \cite[Proposition 1.1]{refHS} is $\alpha_n\beta_n= n^2(n^2+3)$, which does not apply if $n^2+n\neq n^2+3$, that is if $n\neq 3$.
\end{rem}


Next we will show that the symbolic Rees algebra of a Fermat ideal $I$ is  Noetherian. We use the  observation \cite[Theorem 1.3]{refSch} that the Noetherian property of a symbolic Rees algebra is equivalent to the fact that any of its Veronese subalgebras is Noetherian. More precisely, we refer to the subalgebra   
$$\srees(I)^{(n)}:=\srees(I^{(n)})=\bigoplus_{k\geq0} I^{(nk)}$$
as the $n^{\rm th}$ \emph{Veronese subalgebra} of $\srees(I)$. In the case of Fermat ideals, as a corollary of our previous results, we have complete control on the structure of this algebra.

 As an important effect of this, it turns out that the symbolic Rees algebra of $I$  is Noetherian:

 \begin{thm}\label{thm:noeth}
 For any ideal $I$ desribing a Fermat configuration of points, the symbolic Rees algebra  $\srees(I)$ is Noetherian.
 \end{thm}
 
 \begin{proof}
Let $I=(x(y^n-z^n),y(z^n-x^n),z(x^n-y^n))$, with $n\geq 3$. Then $\srees(I)^{(n)}=\rees(I^{(n)})$ by Proposition \ref{prop:Ink}. In particular, $\srees(I)^{(n)}$ is finitely generated. It follows from a result of Schenzel \cite[Theorem 1.3]{refSch} that the symbolic Rees algebra $\srees(I)$ is Noetherian whenever any of its Veronese subrings is Noetherian. In our case, we know that $\srees(I^{(n)})$ is Noetherian, whence the desired conclusion follows.
 \end{proof}
 
 \begin{rem}
 As mentioned in Remark \ref{rem:comp}, Harbourne and Huneke \cite[Proposition 3.5]{refHH} give a condition guaranteeing that a symbolic Rees algebra is Noetherian. In fact,  they wonder \cite[Remark 3.13]{refHH} if this condition is also necessary. Theorem \ref{thm:noeth} shows that this is not the case as $I=(x(y^n-z^n),y(z^n-x^n),z(x^n-y^n))$ does not satisfy the condition if  $n\geq 4$. 
 \end{rem}
 
 \section{Minimal reductions for Fermat ideals} 
 \label{sect5}

Using our detailed knowledge of symbolic powers of Fermat ideals allows us to describe some explicit minimal homogeneous reductions.
 
Let  $J\subset I$ be ideals, then $J$ is said to be a {\em reduction} of $I$ if there exists a non-negative integer $t$ such that $I^{t+1}=JI^t$.  The reduction $J$ is called {\em minimal} if no ideal strictly contained in $J$ is in turn a reduction of $I$.

The minimum integer $n$ with the property $I^{t+1}=JI^t$ for a fixed reduction $J$ of $I$ is called the {\em reduction number} of $I$ with respect to $J$.
In this section we give a description of a homogeneous ideal that is

The following notation will be used in the proof of Proposition \ref{prop:reduction} below: given a homogeneous ideal $I$, the least degree of a non-zero element of $I$ (hence also of a minimal generator of $I$) will be denoted $\alpha(I)$ and the largest degree of a minimal generator of $I$ will be denoted $\beta(I)$.

\begin{prop}
     \label{prop:reduction}
Let $n \ge 3$ be an integer  and consider the ideal $I = (x f, y g, z h)$ of the Fermat configuration, where  $f=y^n-z^n, g=z^n-x^n, h=x^n-y^n \in R=\field[x,y,z]$. Then 
\begin{enumerate}
\item If $n\geq 4$, $I^{(n)}$ has no homogeneous reduction with two generators.
\item A homogeneous minimal reduction of $I^{(n)}$  is 
$$J=\begin{cases}
(fgh,\ g f^2x^n+hg^2y^n+fh^2z^n), &\text{if } n=3\\
(f^{n-2}gh, \ fg^{n-2}h, \ g^{n-2}f^2x^n+h^{n-2}g^2y^n+f^{n-2}h^2z^n), &\text{if } n \ge 4
\end{cases}
$$
and in either case the reduction number of $I$ with respect to $J$ is 1.
\end{enumerate}
\end{prop}

\begin{proof}
$(1)$ Suppose $n\geq 4$ and $J=(\sigma, \tau)$ is a homogeneous minimal reduction for $I$ so that $(I^{(n)})^kJ=(I^{(n)})^{k+1}$ holds for some integer $t\geq 1$, or equivalently, by the identities proven in Proposition \ref{prop:Ink}, $I^{(nk)}J=I^{(n(k+1))}$. Without loss of generality we may assume that $\deg (\sigma)\leq \deg(\tau)$. 

We make the following claims:
\begin{inparaenum}
\item[(i)] $\deg(\sigma)=n^2$,
\item[(ii)] $k=1$.
\end{inparaenum}
To prove the first of these claims, notice that by Theorem \ref{thm:recursion}, $\alpha(I^{(nk)})=n^2k$ and $\alpha(I^{(n(k+1))})=n^2(k+1)$. We must have $\alpha(I^{(nk)}J)=\alpha(I^{(n(k+1))})$, so $n^2k+\deg(\sigma)=n^2(k+1)$, which gives $\deg(\sigma)=n^2$. To prove the second claim we see that $\sigma\in J\subseteq I^{(nk)}$, therefore $\alpha(I^{(nk)})=n^2k\leq \deg(\sigma)=n^2$. It follows that $k=1$ and thus we have $I^{(n)}J=I^{(2n)}$.

It follows from the description of the minimal generators of $I^{(n)}$ and $I^{(2n)}$ of Theorem \ref{thm:gens symb powers} that $(fgh)^2(f,g)^{2(n-3)}\subseteq \sigma \cdot fgh (f,g)^{n-3}$. Comparing the Hilbert function of these two ideals in degree $2n^2$ yields $2(n-3)+1\leq n-2$, i.e $n\leq 3$, which is a contradiction.

$(2)$ is equivalent to showing that $JI^{(n)}=I^{(2n)}$. We prove this statement for  
$$J=(f^{n-2}gh, \ fg^{n-2}h, \ g^{n-2}f^2x^n+h^{n-2}g^2y^n+f^{n-2}h^2z^n),$$
which covers both cases (with some redundancy for $n=3$).
By Remark \ref{rem:compact gens symb powers}, we have
\begin{eqnarray*}
I^{(2 n)} &= &(fgh)^2(f,g)^{2(n-3)} \\
&& + x (x,y)^{n-1} g^{n-2}f^2 \cdot (f^{n-2} g h, \ g^{n-2}f^2 x^n)\\[3pt]
&& + y (y,z)^{n-1} h^{n-2}g^2 \cdot (f g^{n-2} h, \ h^{n-2}g^2 y^n)\\[3pt] 
&& + z (z,x)^{n-1} f^{n-2}h^2 \cdot (g f h^{n-2}, \ f^{n-2}h^2z^n). 
\end{eqnarray*}
The standard minimal generators of the ideal $(fgh)^2(f,g)^{2(n-3)}$ can be written as
$$(fgh)^2f^ig^{2(n-3)-i}=
\begin{cases}
fg^{n-2}h\cdot (fgh)f^{i}g^{n-3-i}& \text{if } 0\leq i \leq n-3\\
f^{n-2}gh\cdot (fgh)f^{i-n+3}g^{2(n-3)i} & \text{if }n-3\leq i\leq 2(n-3),
\end{cases}$$
showing that $(fgh)^2(f,g)^{2(n-3)}\subset(f^{n-2}gh,  fg^{n-2}h)(fgh)(f,g)^{n-3}\subset JI^{(n)}$. Next note that
$$(g^{n-2}f^2x^n+h^{n-2}g^2y^n+f^{n-2}h^2z^n)f^2g^{n-2}x(x,y)^{n-1}\subseteq JI^{(n)}.$$
But 
\begin{eqnarray*}
&(g^{n-2}f^2x^n+h^{n-2}g^2y^n+f^{n-2}h^2z^n)f^2g^{n-2}x(x,y)^{n-1} =\\
 &g^{n-2}f^2x(x,y)^{n-1} \cdot g^{n-2}f^2x^n+ fg^{n-2}h\cdot (fh^{n-3}g^2y^nx(x,y)^{n-1} +f^{n-1}hz^nx(x,y)^{n-1})
\end{eqnarray*}
and the last term in the sum is contained in $JI^{(n)}$, therefore $g^{n-2}f^2x(x,y)^{n-1} \cdot g^{n-2}f^2x^n\subset JI^{(n)}$. Similarly it can be shown that $h^{n-2}g^2y(y,z)^{n-1} \cdot h^{n-2}g^2y^n\subset JI^{(n)}$ and $f^{n-2}h^2z(z,x)^{n-1} \cdot f^{n-2}h^2z^n\subset JI^{(n)}$. The other terms in the description of $I^{(2n)}$ being clearly contained in $JI^{(n)}$, we obtain the containment $I^{2n)}\subseteq JI^{(n)}$. The converse containment being trivial, equality follows.

The fact that $J$ does not contain another homogeneous reduction $L$ for $I^{(n)}$ follows from part $(1)$ of this proposition. A careful reading of the last paragraph in the proof of $(1)$ shows that, if $n\geq 4$, any homogeneous reduction for $I^{(n)}$ must contain at least two generators of degree $n^2$. Hence $(f^{n-2}gh, fg^{n-2}h)\subseteq L$. Since $L$ cannot be 2-generated by $(1)$, it must contain a multiple of the third generator of $J$. Comparing the degrees of the generators of $LI^{(n)}$ and $I^{(2n)}$, one sees that this polynomial must have degree $n^2+n$, the same as the third generator of $J$.  Thus the conclusion $L=J$ follows.
\end{proof}

\begin{rem}
Computational evidence suggests that  each ideal $I^{(n)}$ of Proposition \ref{prop:reduction}, considered in the localization of $R= K[x,y,z]$ at the ideal $(x, y, z)$,  has a minimal reduction generated by two polynomials. This is true for $n=3$ by part (2) of Proposition \ref{prop:reduction}. However, for $n\geq 4$ we have not been able to find a reduction of $I^{(n)}$ in $R$ with only two minimal generators. By part (1) of Proposition \ref{prop:reduction}, such a reduction would necessarily not be  a homogeneous ideal.
\end{rem}


\section*{Acknowledgements}
Computations with {\it Macaulay2} \cite{M2} were essential for discovering and checking the results in this paper. The second author expresses her gratitude to AWM for supporting her through a mentoring travel grant and to the University of Kentucky for hosting her visits during which many of these results were obtained.

\end{document}